\documentclass[12 pt]{article}

\usepackage{amsmath}
\usepackage{amsfonts}
\usepackage{amssymb}
\usepackage{graphicx}
\usepackage[margin=3cm]{geometry}
\usepackage{hyperref}
\usepackage{enumerate}
\usepackage{mathrsfs}
\usepackage{comment}
\usepackage{amsthm}
\usepackage{mathtools}
\usepackage{thmtools, thm-restate}
\newcommand{\R}{\mathbb{R}}
\newcommand{\Q}{\mathbb{Q}}
\newcommand{\Z}{\mathbb{Z}}
\newcommand{\N}{\mathbb{N}}

\DeclareMathOperator{\Diff}{Diff}
\DeclareMathOperator{\Homeo}{Homeo}
\DeclareMathOperator{\Fix}{Fix}

\DeclareMathOperator{\supp}{supp}
\DeclareMathOperator{\Int}{Int}

\DeclareMathOperator{\diam}{diam}
\DeclareMathOperator{\dist}{dist}

\newtheorem{theorem}{Theorem}[section]
\newtheorem{lemma}[theorem]{Lemma}
\newtheorem{proposition}[theorem]{Proposition}
\newtheorem{corollary}[theorem]{Corollary}
\newtheorem{question}[theorem]{Question}
\newtheorem{conjecture}[theorem]{Conjecture}
\newtheorem{example}[theorem]{Example}

\theoremstyle{definition}
\newtheorem{remark}[theorem]{Remark}

\theoremstyle{definition}

\usepackage[usenames,dvipsnames]{color}

\begin{document}

\title{Nilpotent dynamics in dimension one: \\ Structure and smoothness}
\author{Kiran Parkhe}
\date{\today}
\maketitle

\begin{abstract}
Let $M$ be a connected one-manifold, and let $G$ be a finitely-generated nilpotent group of homeomorphisms of $M$. Our main result is that one can find a collection $\{I_{i, j}, M_{i, j}\}$ of open disjoint intervals with dense union in $M$, such that the intervals are permuted by the action of $G$, and the restriction of the action to any $I_{i, j}$ is trivial, while the restriction of the action to any $M_{i, j}$ is minimal and abelian.

It is a classical result that if $G$ is a finitely-generated, torsion-free nilpotent group, then there exist faithful continuous actions of $G$ on $M$. Farb and Franks \cite{F&F} showed that for such $G$, there always exists a faithful $C^1$ action on $M$. As an application of our main result, we show that \emph{every} continuous action of $G$ on $M$ can be conjugated to a $C^{1 + \alpha}$ action for any $\alpha < 1/d(G)$, where $d(G)$ is the degree of polynomial growth of $G$.
\end{abstract}

\section{Introduction}

In this paper, we study homeomorphisms and diffeomorphisms of one-manifolds, i.e., $\R, [0, 1), [0, 1]$, or $S^1$. For convenience, we will always assume they are orientation-preserving. From a dynamical standpoint, a homeomorphism of the line is not very interesting. It will have some closed (possibly empty) set of fixed points, on the complement of which points wander to the left or to the right.

The dynamics of a homeomorphism of the circle \emph{is} interesting; this theory is due to Poincar\'{e}. He defined a quantity called the \emph{rotation number} $\rho(f) \in \R/\Z$ of a homeomorphism $f$, and showed that this quantity contains valuable dynamical information about $f$. Namely, if $\rho(f)$ is zero (rational), $f$ has fixed (periodic) points. On the other hand, if $\rho(f)$ is irrational, $f$ is topologically semi-conjugate to a rotation of the circle by the angle $\rho(f)$. In this situation there are two sub-possibilities: (1) $f$ is in fact conjugate to a rotation, so its dynamics is minimal (every orbit is dense), or (2) there is a Cantor set on which $f$ is minimal, the complement of which consists of ``wandering intervals'' that never return to themselves. See \cite{Ghys} for more information.

A homeomorphism of a manifold $M$ can be viewed as giving an action of $\Z$ on $M$, via $n \mapsto f^n$. In this light, it is natural to consider actions of other groups, such as abelian groups, or more generally nilpotent groups. Formally, by an \emph{action} of $G$ on $M$ we mean a homomorphism $\phi\colon G \to \Homeo_+(M)$, the group of orientation-preserving homeomorphisms of $M$. (We deal only with discrete groups in this paper, so continuity of $\phi$ is not an issue.) We say the action is \emph{faithful} if $\phi$ is one-to-one.

For example, one can get faithful $\Z^n$ actions on $S^1$ (resp. $\R$) by taking $n$ rationally independent rotations (resp. translations). Notice that these actions will be minimal for any $n > 0$ on $S^1$, and for any $n > 1$ on $\R$.

There is a close interplay between group actions on the circle and on the line. First, given an abelian action on the circle with a \emph{global fixed point} (point fixed by the whole group), one can remove the fixed point, and what remains is an action on the line. Conversely, given an action on $\R$ we can add a point at infinity to get an action on $S^1$.

More interestingly, suppose we are given $n$ commuting homeomorphisms $f_1, \ldots, f_n$ of the line. It turns out that if there are no global fixed points for the action of $f_1, \ldots, f_n$, then some $f_i$ must have no fixed points. We can quotient $\R$ by the action of $f_i$, yielding a circle. Since the other homeomorphisms commute with $f_i$, they induce homeomorphisms $\bar{f_j}$ of the quotient circle $\R/f_i$. Thus, an action of $\Z^n$ on $\R$ with no global fixed points yields an action of $\Z^{n - 1}$ on $S^1$. The converse is also true: given an action of $\Z^{n - 1}$ on $S^1 = \R/\Z$, we can lift it to a group of $n - 1$ homeomorphisms of $\R$ commuting with translation by one, giving a $\Z^n$ action on $\R$.

It is a folklore theorem that any finitely-generated torsion-free nilpotent group admits faithful actions by homeomorphisms of the line (and hence also the circle). This is equivalent to the fact that such a group admits left-invariant total orders \cite{Witte}. Explicitly, following Farb and Franks \cite{F&F} we can construct actions of these groups as follows. Malcev's work \cite{Malcev} (see also \cite{Raghunathan}) shows that every finitely-generated torsion-free nilpotent group $G$ embeds in some $L_n = \{n \times n \text{ lower-triangular matrices with 1s on the diagonal}\}$. Thus, to define an action of $G$ by homeomorphisms of the line, we need only define an action of $L_n$.

We take lower-triangular matrices because they are convenient in the following way: the standard action of $L_n$ on $\Z^n$ (on the left) preserves the lexicographic order on $\Z^n$. It is possible to choose a collection of disjoint open subintervals of $\R$ with dense union in $\R$, indexed by $\Z^n$ and with the property that $I_{(a_1, \ldots, a_n)}$ is to the left of $I_{(b_1, \ldots, b_n)}$ if and only if $(a_1, \ldots, a_n) < (b_1, \ldots, b_n)$ in the lexicographic order. We can define an action of $L_n$ on $\bigcup_{\vec{a} \in \Z^n} I_{\vec{a}}$ by letting $M \in L_n$ send $I_{\vec{a}}$ to $I_{M\vec{a}}$, say in an affine way. Since we chose the $I_{\vec{a}}$ to have dense union in $\R$, these maps extend uniquely to homeomorphisms of the line.

In the examples of nilpotent actions we have seen so far, there have been two basic types of behavior. One is minimality, as in at least one irrational rotation of the circle or at least two independent translations of the line. The other is discrete intervals with dense union being permuted in some way. This was the case for the general construction of nilpotent actions we gave, and the reader can check it also holds for any single homeomorphism of a one-manifold not conjugate to an irrational rotation of the circle. Our main result says that these two behaviors, in some combination, are all we can see for nilpotent actions on one-manifolds.

\begin{restatable}[Structure Theorem]{theorem}{StructureTheorem}
\label{StructureTheorem}
Let $M$ be a one-manifold, and $G \subset \Homeo_+(M)$ be a finitely-generated virtually nilpotent group. There exist (countably many) open sets $I_i$ and $M_i$ such that the following hold:

\begin{itemize}
\item $I_i \cap I_j = M_i \cap M_j = \emptyset$ unless $i = j$, and $I_i \cap M_j = \emptyset$ for all $i, j$.

\item Each $I_i$ and $M_i$ is $G$-invariant.

\item Let $I_{i, j}$ be the open intervals comprising $I_i$. For any $j, j'$, there exists $g \in G$ such that $g(I_{i, j}) = I_{i, j'}$. If $g(I_{i, j}) = I_{i, j}$, then $g|_{I_{i, j}} = id|_{I_{i, j}}$.

\item For any $i$, the action of $G$ on $M_i$ is minimal. If $M = S^1$ and the action of $G$ on $S^1$ is minimal, $G$ is abelian. If some $M_i$ is composed of open intervals $M_{i, j}$ then the group $G_{i, j} = \{g|_{M_{i, j}}\colon g \in G, g(M_{i, j}) = M_{i, j}\}$ is abelian.

\item $\bigcup_i I_i \cup \bigcup_j M_j \subset M$ is dense.
\end{itemize}
\end{restatable}

In Figure 1, we show an example of what this could look like for an action of the Heisenberg group $H = \langle f, g, h\colon [f, g] = h \text{ and } [f, h] = [g, h] = id\rangle$ on $\R$. The action of $f$ is by translation, as shown. The elements $g$ and $h$ act minimally on the wavy intervals, while in the other intervals, $h$ acts trivially and a $g$-orbit is depicted. For this action, there is only one $I_1$ and only one $M_1$, as shown. Note that the action of $H$ on $M_1$ is \emph{not} abelian, but its restriction to any of the intervals in $M_1$ is.

There two key nontrivial assertions in our theorem, which depend on fact that finitely-generated virtually nilpotent groups have polynomial \emph{growth} (see below for the definition). The first is that, on a $G$-invariant interval on which the action is minimal, the action is abelian. This assertion holds for any finitely-generated group of sub-exponential growth. However, it is not true for solvable groups. For instance, consider the group of homeomorphisms of $\R$ generated by $f(x) = x + 1$ and $g(x) = 2x$. This action is minimal: the orbit of $0$ contains every dyadic rational. But clearly, the group is not abelian.

The second nontrivial assertion is that, in regions where the action is not minimal, there are no dense orbits. This fails in general for groups not having polynomial growth. Grigorchuk-Maki \cite{G&M} gave an example of a group $G$ of intermediate (super-polynomial, sub-exponential) growth acting faithfully by homeomorphisms of the line. This group action has dense orbits, but also has (a dense set of) orbits that are not dense. We show three of these orbits in Figure 2. All the dots of a given size lie in the same orbit; for clarity, we have shown only part of the densest orbit.

\begin{figure}
  \centering
    \includegraphics[width=4in]{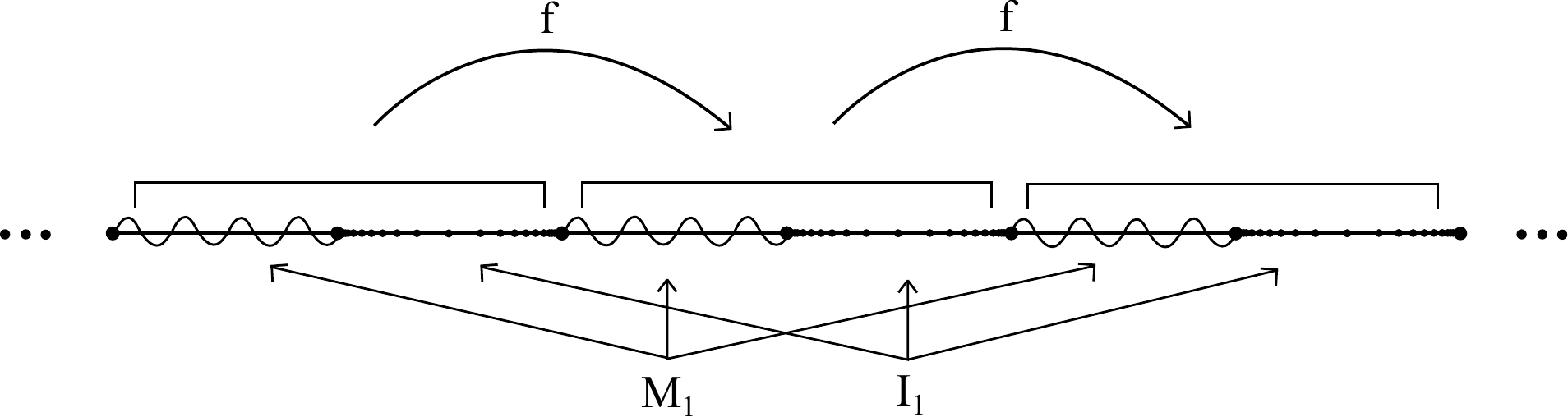}
    \caption{Action of the Heisenberg group on the line}
\vspace{12pt}
\includegraphics[width=3in]{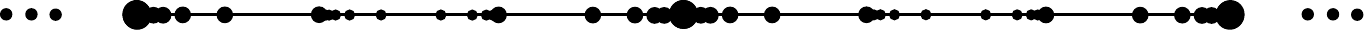}
   \caption{Orbits of the Grigorchuk-Maki group action on the line}
\end{figure}

This second assertion is proved for finitely-generated virtually nilpotent groups by an inductive procedure, looking at finer and finer intervals. The procedure must end after finitely many steps, since these groups have polynomial growth, and each time we restrict to a finer interval the growth degree becomes smaller.

The inductive procedure yielding the desired intervals is a conjugacy invariant: if we have two topologically conjugate actions of a finitely-generated virtually nilpotent group $G$, then the structure of the intervals we get will agree. We have the following immediate corollary:

\begin{corollary}
Let $M$ be a 1-manifold, and $G$ a finitely-generated virtually nilpotent group. Two actions $\phi, \psi\colon G \to \Homeo_+(M)$ are conjugate if and only if there is a homeomorphism of $M$ sending $(I_{i, j})_\phi$ and $(M_{i, j})_\phi$ to $(I_{i, j})_\psi$ and $(M_{i, j})_\psi$ and intertwining the $\phi$ and $\psi$ actions of $G$ on these intervals. In particular, the action of $G$ on $(M_{i, j})_\phi$ by $\phi$ and the action of $G$ on $(M_{i, j})_\psi$ by $\psi$ are conjugate groups of translations whose elements must have the same translation numbers relative to each other.
\end{corollary}

\subsection{Application: Smoothing $C^0$ to $C^1$}

Let $M$ be a manifold, and $f\colon M \to M$ a homeomorphism. It is natural to ask, what is the smoothest diffeomorphism of $M$ topologically conjugate to $f$? On surfaces, it is easy to construct homeomorphisms not conjugate to $C^1$ diffeomorphisms. For example, let $f$ have a fixed point $x$, and rotate points faster and faster as one approaches $x$, with the amount of rotation tending to infinity. No matter what conjugacy one applies, there will still be a ``bad'' point with infinitely fast rotation, so the result will not be a $C^1$ diffeomorphism.

In fact, Harrison \cite{Harrison} has shown the following: on any $n$-manifold ($n \neq 1$ or $4$), there exist $C^r$ diffeomorphisms not conjugate to $C^s$ diffeomorphisms for any real numbers $s > r > 0$ (if these are not integers, the fractional parts are interpreted as H\"{o}lder exponents). She shows moreover that if $M$ is compact, such \emph{unsmoothable} diffeomorphisms are dense in $\Diff^r(M)$ in the $C^0$ topology.

In dimension one the possibilities for dynamics, and correspondingly for unsmoothability, are much more limited. The line $\R$ is particularly simple: as we have seen, a homeomorphism $f\colon \R \to \R$ is characterized by its set of fixed points and the direction points in $\R \setminus \Fix(f)$ move. One can easily check, using this, that any homeomorphism of $\R$ is topologically conjugate to a $C^\infty$ diffeomorphism; one simply needs to use gentle enough ``bump functions'' in the intervals of $\R \setminus \Fix(f)$.

The circle $S^1$ is not quite so simple, since there can be nontrivial recurrence. From the point of view of smoothability, there are two possibilities. The first is that the homeomorphism $f\colon S^1 \to S^1$ either has rational rotation number or is conjugate to an irrational rotation. In this case, $f$ is conjugate to a $C^\infty$ diffeomorphism of $S^1$. The second is that $f$ is semi-conjugate, but not conjugate, to an irrational rotation (i.e., it has an exceptional minimal set). In this case, by the following theorem of Denjoy, $f$ cannot be topologically conjugated to $C^{1 + bv}$:

\begin{theorem}[Denjoy \cite{Denjoy}]
If $f\colon S^1 \to S^1$ is a $C^{1 + bv}$ diffeomorphism with irrational rotation number $\theta$, then $f$ is topologically conjugate to the rotation $R_\theta$.
\end{theorem}

In this paper, we show that a homeomorphism of the second type is still always conjugate to a $C^{1 + \alpha}$ diffeomorphism for any $\alpha < 1$. Thus there is a threshold of regularity for circle diffeomorphisms; one class of diffeomorphisms can be smoothed only up to this threshold, while the other can be smoothed past it. More generally, we prove the following, with the help of Theorem \ref{StructureTheorem}:

\begin{theorem}
\label{Smoothability}
Let $M$ be a 1-manifold, and let $G \subset \Homeo_+(M)$ be a finitely-generated, virtually nilpotent subgroup. Then there exists a homeomorphism $\psi\colon M \to M$ such that $\psi G \psi^{-1} \subset \Diff_+^{1 + \alpha}(M)$ for any $\alpha < 1/d$, where $d$ is the degree of polynomial growth of $G$.
\end{theorem}

This result strengthens the following theorem of Farb and Franks \cite{F&F}:

\begin{theorem}
Let $M$ be a 1-manifold. Every finitely-generated, torsion-free nilpotent group is isomorphic to a subgroup of $\Diff_+^1(M)$.
\end{theorem}

Farb and Franks's theorem says that for finitely-generated torsion-free nilpotent groups there \emph{exist} faithful actions; our theorem says that \emph{every} action can be conjugated to $C^1$ (and a little more). The following comes easily from our theorem:

\begin{corollary}
\label{Corollary}
Let $G \subset \Homeo_+(S^1)$ be finitely-generated and virtually nilpotent. Given any fixed generating set for $G$, it is possible to conjugate $G$ to a subgroup of $\Diff_+^1(S^1)$ such that the generators are as $C^1$ close to rotations as desired. The same is true for $G \subset \Homeo_+(\R)$, with ``rotations'' replaced by ``the identity.''
\end{corollary}

This gives another method for proving Theorem B from \cite{Navas4}.

\subsection{Remarks and Questions}
In this section, we explore the extent to which our theorem is sharp. We first observe that we cannot in general conjugate the group action to $C^2$, even if $M = \R$. To see this, take a homeomorphism $f$ of the circle with exceptional minimal set. Lift $f$ to a homeomorphism $\tilde{f}\colon \R \to \R$ commuting with translation by 1, $T_1$. The group generated by $\tilde{f}$ and $T_1$ cannot be conjugated to $C^2$.

Next, we note that our theorem cannot be strengthened to include the class of solvable groups. It is shown in \cite{C&C} and \cite{G&L} that there exist actions of the Baumslag-Solitar group $BS(1, 2) = \langle a, b \colon aba^{-1} = b^2\rangle$ on $S^1$ which cannot be conjugated to $C^1$. This example makes use of the fact that $b$ is exponentially distorted (see below) in $BS(1, 2)$. Bonatti, Monteverde, Navas, and Rivas \cite{BMNR} give a large class of solvable examples (including this one) that cannot be conjugated to $C^1$.

By Gromov's Theorem \cite{Gromov2}, our results apply to any group of polynomial growth. It would be interesting to understand actions of groups with sub-exponential growth. In particular, it seems natural to ask:

\begin{question}
Is the Grigorchuk-Maki group action on the line conjugate to a $C^1$ action?
\end{question}

Navas \cite{Navas2} has shown that it is \emph{semi-conjugate} to a $C^1$ action: by adding ``buffer'' regions on which the action is trivial, he constructs a faithful $C^1$ action of this group.

Another necessary assumption in our results is finite generation of the group. It is easy to construct uncountable abelian groups of homeomorphisms of $\R$ which are not conjugate to Lipschitz, let alone $C^1$. On the other hand, a remarkable result of Deroin, Kleptsyn, and Navas \cite{DKN} says that any countable group of homeomorphisms in dimension one is conjugate to a group of bi-Lipschitz homeomorphisms. This leads us to the following question:

\begin{question}
If $M$ is a 1-manifold, and $G \subset \Homeo_+(M)$ is a countable nilpotent (or even abelian) subgroup, is $G$ conjugate to a subgroup of $\Diff_+^1(M)$?
\end{question}

%\begin{remark}
%The finite generation hypothesis in Theorem \ref{Smoothability} is also necessary. For example, consider the following. Let $\phi_0\colon \R \to (-1/2, 1/2)$ be given by $\phi_n(x) = \frac{1}{\pi}\arctan(x) + n$. Now define the following homeomorphisms of $\R$: $f_1(x) = x + 1$, and for $i > 1$,

%\[f_i(x) =  \left\{
%\begin{array}{ll}
%      \phi_nf_{i - 1}\phi_n^{-1}(x), & x \in (n - 1/2, n + 1/2) \\
%      x, & x \in \Z + 1/2
%\end{array} 
%\right. \]

%Let $H$ be the group generated by the $f_i$. Clearly, $H$ is abelian. Note that the maximum distance a point moves under $f_i$ goes to 0 as $i \to \infty$, since $\phi_n$ has derivative at most $\frac{1}{\pi}$. Using this, the reader can check that an infinite ``word'' of the form $\ldots \circ f_2^{n_2}\circ f_1^{n_1}$ will yield a homeomorphism of the line, regardless of how the $n_i$ are chosen. Thus we can enlarge $H$ to the uncountable abelian group $G$ consisting of such homeomorphisms. Note that the homeomorphisms in $G$ are 1-periodic.

%If $G$ were conjugate to a subgroup of $\Diff_+^1(\R)$, it would be possible to make the generators $f_i$ ``gentle'' enough that for any choice of $n_i$, the map $\ldots \circ f_2^{n_2}\circ f_1^{n_1}$ is $C^1$. We claim this is not the case. No matter what the generators $f_i$ are after conjugacy, if we make the $n_i$ increase fast enough in $i$, the map $\ldots \circ f_2^{n_2}\circ f_1^{n_1}$ will not be Lipschitz continuous, so it cannot be $C^1$.

%We are not sure if the main result holds for countable virtually nilpotent subgroups of $\Homeo_+(M)$.
%\end{remark}

A final line of questions involves the optimal regularity for finitely-generated nilpotent actions. We show in this paper that any finitely-generated virtually nilpotent subgroup $G \subset \Homeo_+(M)$ is conjugate to a subgroup of $\Diff_+^{1 + \alpha}(M)$ for any $\alpha$ less than the degree of polynomial growth of $G$. This leaves open the following questions (the first one due to Navas):

\begin{question}
For a finitely-generated virtually nilpotent group $G$, what is the smoothest faithful action of $G$ on [0, 1]? What is the least smoothable action?
\end{question}

Let $N_n$ ($n \geq 3$) be the group of $n \times n$ upper-triangular matrices with 1s on the diagonal. Farb and Franks \cite{F&F} exhibited a $C^1$ action of this group on $[0, 1]$; Castro, Jorquera, and Navas \cite{CJN} showed that this action can be smoothed to any differentiability class less than $C^{1 + \frac{2}{(n - 1)(n - 2)}}$ (and subsequently Navas \cite{Navas3} showed $C^{1 + \frac{2}{(n - 1)(n - 2)}}$ is not possible). Note that the group $N_n$ has growth degree $d(N_n) = \frac{(n - 1)(n)(n + 1)}{6}$, so $1/d(N_n) < \frac{2}{(n - 1)(n - 2)}$. Thus, for these groups one can always achieve greater smoothness than our results suggest. In fact, Jorquera \cite{Jorquera} showed that there exist groups of arbitrarily high nilpotence degree that embed in $\Diff_+^{1 + \alpha}([0, 1])$ for any $\alpha < 1$.

\vspace{12pt}
\noindent\textbf{Acknowledgments.} I would like to thank Andr\'{e}s Navas for referring me to several important articles, including the one of Cantwell-Conlon \cite{C&C} and the one of Guelman-Liousse \cite{G&L}, for pointing out Corollary \ref{Corollary} and its relation to his article \cite{Navas4}, and for his enthusiasm for this research. I would like to thank Tobias Hartnick and Uri Bader for helpful conversations. I am grateful for the support of Uri Shapira and a fellowship at the Technion from the Lady Davis Foundation.

\section{Proof of Theorem \ref{StructureTheorem} and Theorem \ref{Smoothability}}
%The structure of the argument is as follows. First, we recall that if $G$ is a finitely-generated virtually nilpotent subgroup of $\Homeo_+(M), M = S^1$ or $\R$, then $[G, G]$ must have a global fixed point. This is an ingredient in a structure theorem (Theorem \ref{StructureTheorem}) we derive for finitely-generated virtually nilpotent subgroups of $\Homeo_+(M)$, $M$ a 1-manifold, which says that up to a nowhere dense closed set, the manifold can be split into intervals $I$ such that for every $g \in G, g(I) = I$ or $g(I) \cap I = \emptyset$, and if $G_I = \{g|_I \colon g \in G, g(I) = I\}$, $G_I$ is either trivial or abelian and in the latter case acts minimally on $I$. We then use a technique found in \cite{F&F}, \cite{Navas} and due partly to a suggestion of Yoccoz to get functions between these intervals having nice differentiability properties. Finally, we define the lengths of the intervals carefully (in a way that depends on the \emph{growth rate} of $G$) to force the action of $G$ to be $C^1$.

The next two results are involved in showing that, for nilpotent group actions, minimal implies abelian. The following proposition is easy. See e.g. \cite{Navas}.

\begin{proposition}
If $G \subset \Homeo_+(S^1)$ is any amenable subgroup (not necessarily finitely-generated), there a $G$-invariant Borel probability measure $\mu$. Therefore, rotation number is a homomorphism on $G$, so each element of $[G, G]$ has a fixed point, and for any $x\in \supp(\mu)$, $x \in \cap_{g \in [G, G]} \Fix(g)$ -- in particular, $[G, G]$ has a global fixed point.
\end{proposition}

Define a \emph{Radon measure} on $\R$ to be a Borel measure which is finite on compact subsets (note that the other usual requirements for a Radon measure are automatically satisfied on $\R$). The following result is due to Plante \cite{Plante}; see also \cite{Navas}, Theorem 2.2.39.

\begin{theorem}
\label{RadonMeasure}
If $G$ is a finitely-generated virtually nilpotent subgroup of $\Homeo_+(\R)$, there is a $G$-invariant Radon measure on the line. Therefore, translation number is a homomorphism on $G$, so each element of $[G, G]$ has a fixed point, and for any $x\in \supp(\mu)$, $x \in \cap_{g \in [G, G]} \Fix(g)$ -- in particular, $[G, G]$ has a global fixed point.
\end{theorem}

Note that Theorem \ref{RadonMeasure} holds for any finitely-generated group of sub-exponential growth (for the definition of growth, see below), but does not hold for solvable groups in general.

It is an easy fact that for any $M$ and $G \subset \Homeo(M)$, $\Fix([G, G])$ is a $G$-invariant set. We can see this as follows. Let $x \in \Fix[G, G]$ and $g \in G$; we want to show that for any $h \in [G, G]$, $g(x) \in \Fix(h)$. But $[G, G]$ is a normal subgroup of $G$, so $g^{-1}hg \in [G, G]$; thus, $g^{-1}hg(x) = x$, so $hg(x) = g(x)$, as desired.

Before proving Theorem \ref{StructureTheorem}, we need the following notions. Let $f, g\colon \N \to \N$ be monotone increasing. We will write $f \lesssim g$ if there exists a constant $0 < C \in \N$ such that $f(n) \leq Cg(Cn)$ for all $n \in \N$. We will write $f \sim g$ if $f \lesssim g$ and $g \lesssim f$. This defines an equivalence relation. If $[f]$ and $[g]$ are equivalence classes, we may write $[f] \leq [g]$ or $[f] = [g]$ if $f \lesssim g$ or $f \sim g$, respectively. By abuse of notation, $[f] = g$ will mean $[f] = [g]$.

For the moment, let $G$ be any finitely-generated group. Let $S = \{g_1, \ldots, g_k\}$ be a symmetric finite generating set for $G$. This enables us to define the \emph{norm} $|g|_S$ for $g \in G$ as $$|g|_S = \min\{n \geq 0\colon \text{there exist } g_{i_1}, \ldots, g_{i_n} \in S \text{ such that } g = g_{i_1}\cdots g_{i_n}\}.$$

We define the \emph{growth function} of $G$ with respect to the generating set $S$ to be $$\mathscr{G}_{(G, S)}(n) = \#\{g \in G\colon |g|_S \leq n\}.$$ Though this function depends on the choice of generating set, if $S'$ is any other choice we will have $\mathscr{G}_{(G, S)}(n) \sim \mathscr{G}_{(G, S')}(n)$, so we can talk unambiguously about the equivalence class of these, which we will denote $\mathscr{G}_G(n)$.

It is immediate that if $H$ is a quotient of $G$, $\mathscr{G}_H(n) \leq \mathscr{G}_G(n)$. It is also easy to see that if $H \subset G$ is a finitely-generated subgroup then $\mathscr{G}_H(n) \leq \mathscr{G}_G(n)$. We say that $G$ has \emph{polynomial growth} if for some integer $d(G)$, $\mathscr{G}_G(n) = n^{d(G)}$.

Finitely-generated virtually nilpotent groups have polynomial growth. We have the following theorem, apparently derived independently by Guivarc'h \cite{Guiva}, Bass \cite{Bass}, and others:

\begin{theorem}
Let $N$ be a finitely-generated nilpotent group, and $N^i$ the $i^{th}$ subgroup in the lower central series for $N$ (with $N^1 = N$). Let $r_i = rank(N^i/N^{i + 1})$ (the torsion-free rank). Then $d(N) = \Sigma_{i \geq 1} i\cdot r_i$.
\end{theorem}

The quantity $d(N)$ is also referred to as the \emph{homogeneous dimension} of $N$; see e.g. \cite{delaHarpe}. If $G$ is a finitely-generated virtually nilpotent group, and $N$ is a (necessarily finitely-generated) nilpotent subgroup of finite index, then $d(G) = d(N)$.

It is a fact, not hard and proved in Bass \cite{Bass}, is that if $G$ is a finitely-generated nilpotent group and $H \subset G$ a subgroup of infinite index, the $d(H) < d(G)$. Of course, the same holds if $G$ is virtually nilpotent and $H \subset G$ infinite index. Indeed, if $N \subset G$ is a finite-index subgroup, then $N\cap H$ has finite index in $H$, so $d(N\cap H) = d(H)$; also, $N\cap H$ has infinite index in $N$, so by Bass, $d(N\cap H) < d(N) = d(G).$

%Now let $G$ be our given subgroup of $\Homeo_+(M)$. Let $A_i \subset G$ be the subgroup of $G$ sending $I_{i, 0}$ to itself. Thus if $g, h\colon I_{i, 0} \to I_{i, j}$, then $h = ga$ for some $a \in A_i$; the set of group elements sending $I_{i, 0}$ to $I_{i, j}$ is a left coset of $A_i$.

%\begin{definition}
%A theorem of Malcev states that there is a nilpotent Lie group $\mathfrak{G}$ and an embedding $\phi\colon G \to \mathfrak{G}$ such that $\phi(G)$ is discrete and cocompact in $\mathfrak{G}$; furthermore, if $\mathfrak{G}'$ and $\phi '$ are another such Lie group and embedding, then there is a unique Lie group isomorphism $\psi\colon \mathfrak{G} \to \mathfrak{G}'$ such that $\psi\circ\phi = \phi '$. We may define the \emph{dimension} $\dim(G)$ to be the dimension of $\mathfrak{G}$.
%\end{definition}

\begin{proof}[Proof of Theorem \ref{StructureTheorem}]
Let us write $\mathcal{I}$ for the set of $I_i$, and $\mathcal{M}$ for the set of $M_i$. We have a bit of freedom in constructing $\mathcal{I}$, but no control over $\mathcal{M}$. Indeed, if $M_i$ and $M_j$ are minimal open sets for the action of $N$, we claim they are equal or disjoint. Suppose neither of these holds; then we can find $x \in M_i \cap M_j$, and without loss of generality $y \in M_i \setminus M_j$. A small open neighborhood of $x$ must be contained in $M_i \cap M_j$, since $M_i \cap M_j$ is open. Since the orbit of $y$ is dense in $M_i$ and $x \in M_i$, the orbit of $y$ must enter this small neighborhood of $x$, and hence enter $M_i \cap M_j$. But since $M_j$ is a $G$-invariant set and $y \notin M_j$, the orbit of $y$ cannot enter $M_j$, a contradiction.

First suppose that $M = S^1$, and that the action of $G$ is minimal. We claim that $G$ must be abelian. By Proposition 2.1, $\Fix([G, G])$ is nonempty. Since this set is $G$-invariant, by minimality it is dense in $S^1$; therefore, $\Fix([G, G]) = S^1$. Thus $G$ is abelian.

If the action of $G$ on $S^1$ is not minimal, let $x \in S^1$ be a point whose $G$-orbit is not dense. Since $\overline{\mathcal{O}_G(x)}$ is a $G$-invariant set, we may restrict attention to the open intervals in $S^1\setminus \overline{\mathcal{O}_G(x)}$. If, given an arbitrary maximal open interval $I \subset S^1\setminus \overline{\mathcal{O}_G(x)}$, we can find the desired sets $I_i$ and $M_i$ for the group $G_I = \{g|_I \colon g \in G, g(I) = I\}$, then we can certainly find these sets for the $G$-action on $S^1$. We may therefore assume without loss of generality that the manifold on which $G$ acts is an interval, say $M = \R$.

Suppose that $M_i$ is some minimal open $G$-set, and $I \subset M_i$ is a maximal open interval. Again, let $G_I = \{g|_I \colon g \in G, g(I) = I\}$. As above, by applying Theorem \ref{RadonMeasure}, we can see that $G_I$ is abelian.

If $\bigcup_i M_i \subset \R$ is dense, then the theorem is proved. Therefore, we assume that $\R \setminus \bigcup_i M_i$ has nonempty interior. Let $I \subset \R \setminus \bigcup_i M_i$ be an arbitrary maximal open interval. Let $G_I$ be as before. If we can find a set of intervals $I_i$ for the action of $G_I$ on $I$ having the requisite properties (including dense union in $I$), then we automatically get the desired intervals for the action of $G$ on the $G$-orbit of $I$. Therefore, it suffices to restrict attention to $I$; we may assume $G$ acts on $\R$ with no open minimal sets.

We will show, by induction on $d(G)$, that for a group $G$ of homeomorphisms of $\R$ with no open minimal subsets we can find the desired permuted intervals. If $d(G) = 0$, then since $G\subset \Homeo_+(\R)$, $G$ is torsion-free, and therefore trivial. Thus the result is trivial in this case; we let $\mathcal{I}$ simply contain $\R$.

If $d(G) = 1$, it is not hard to see that $G \cong \Z$. Indeed, $G$ contains a nilpotent subgroup of finite index $N$ such that $d(N) = 1$, and a torsion-free nilpotent group with $d(N) = 1$ must be isomorphic to $\Z$ (any larger nilpotent group would contain a copy of $\Z^2$, and $d(\Z^2) = 2$). Thus $G$ is a finite torsion-free extension of a group isomorphic to $\Z$, so $G$ itself is isomorphic to $\Z$. Therefore, $G$ is generated by a single element $g$. If $\Fix(g)$ has nonempty interior, we may add the maximal open intervals in $\Fix(g)$ to $\mathcal{I}$. Thus we may remove the fixed points of $g$, and show that we get the desired interval structure on the remaining open intervals $I$. Let $x$ be an arbitrary element of such an $I$, and let $I_n$ be the interval $(g^n(x), g^{n + 1}(x))$. This gives us the desired interval structure.

Now let $d(G)$ be arbitrary, and assume the result holds for finitely-generated virtually nilpotent subgroups of $\Homeo_+(\R)$ with no open minimal regions having growth degree smaller than that of $G$. We may remove any global fixed points of the action of $G$ on $\R$, putting any open intervals contained therein into $\mathcal{I}$. Thus we may assume that $G$ has no global fixed points. Obviously, $(\R \setminus \Fix([G, G])) \cup \Int(\Fix([G, G])) \subset \R$ is dense, and these two sets are $G$-invariant. Therefore, it suffices to find the desired intervals in $\R \setminus \Fix([G, G])$ and in $\Int(\Fix([G, G]))$.

To find the necessary intervals for $\R \setminus \Fix([G, G])$, it suffices to do this for an arbitrary maximal open interval $I \subset \R \setminus \Fix([G, G])$. As before, form $G_I$ by taking the subgroup of $G$ sending $I$ to itself, followed by the quotient of this subgroup that we get by restriction to $I$. We have taken a quotient of a subgroup of infinite index, since if $g(I) \neq I$ then $g^n(I) \neq I$ for all $n > 0$. Therefore, as noted above, we have that $d(G_I) < d(G)$. So by inductive hypothesis we can find the desired intervals for the action of $G_I$ on $I$.

Similarly, if $I$ is a maximal open interval in $\Int(\Fix([G, G]))$ then the group $G_I$ satisfies $d(G_I) < d(G)$, completing the argument by induction, \emph{unless} $G$ is abelian (so $I$ is all of $\R$). In that case we have $G \cong \Z^k$ for some $k > 1$. We have assumed that $G$ has no global fixed points, so some $f \in G$ is conjugate to a translation and we can consider $\bar{G} = G/f$ to be a subgroup of $\Homeo_+(S^1)$ isomorphic to $\Z^{k - 1}$. $\bar{G}$ does not act minimally on $S^1$, or else $G$ would have acted minimally on $\R$. Therefore, $\bar{G}$ has a minimal invariant closed set $X \subset S^1$ which is either finite or a Cantor set. For $I$ a maximal open interval in $S^1 \setminus X$, we can take $\bar{G}_I$, the quotient of a subgroup of $\bar{G}$ acting on $I$. We have $d(\bar{G}_I) \leq k - 1 < k$, so by inductive hypothesis we can find the desired interval structure for the action of $\bar{G}_I$ on $I$ and hence for $G$ on $\R$.
\end{proof}

Now our goal is to prove Theorem \ref{Smoothability} with the aid of Theorem \ref{StructureTheorem}. We will assume that we are not in the situation of a minimal abelian action on $S^1$, because actions of that type can always be conjugated to a group of rotations. Therefore, we have intervals $I_{i, j}$ and $M_{i, j}$; we must construct a conjugacy to adjust the $G$-action on these intervals and the lengths of the intervals so that the result is $C^{1 + \alpha}$.

It will be useful to us to define a family of diffeomorphisms $\phi_{a, b}\colon (-a/2, a/2) \to (-b/2, b/2), a, b \in \R_{> 0}$ which is \emph{equivariant}, i.e. such that $\phi_{b, c}\circ\phi_{a, b} = \phi_{a, c}$. (Thus we have the structure of a groupoid, whose objects are intervals and whose morphisms are the diffeomorphisms.) Following Farb-Franks \cite{F&F} and Navas \cite{Navas}, who credit J. C. Yoccoz for the idea, for any $a > 0$ we define $\phi_a\colon \R \to (-a/2, a/2)$ by $\phi_a(x) = \frac{a}{\pi}\arctan(ax)$. Then we define $\phi_{a,b}\colon (-a/2, a/2) \to (-b/2, b/2)$ by $\phi_{a, b} = \phi_b\phi_a^{-1}$.

It is immediate from this definition that the diffeomorphisms $\phi_{a, b}$ are equivariant. They have very well-behaved differentiability properties, namely that $\phi_{a, b}'(-a/2) = \phi_{a, b}'(a/2) = 1$, and $\phi_{a, b}'$ is uniformly close to 1 provided that $b/a$ is close to 1. In fact, Navas (\cite{Navas}, Lemma 4.1.25) shows that for any H\"{o}lder exponent $\alpha$, if $a > 0$ and $b > 0$ are such that $a/b \leq 2$, $b/a \leq 2$, and $$\left|\frac{b}{a} - 1\right|\cdot \frac{1}{a^\alpha} \leq C,$$ then we get a bound on the $C^\alpha$ \emph{norm} of $\phi_{a, b}'$:

\begin{equation}
 |\phi_{a, b}'|_{C^\alpha} \coloneqq \sup_{x, y \in (-a/2, a/2)} \frac{|\phi_{a, b}'(y) - \phi_{a, b}'(x)|}{|y - x|^\alpha} \leq 6\pi C.
\end{equation}

%More precisely, Navas \cite{Navas} shows that when $\alpha \geq \beta$, $$\sup_{x \in (-\alpha/2, \alpha/2)} |1 - \phi_{\alpha, \beta}'(x)| = 1 - \frac{\beta^2}{\alpha^2}.$

It is these properties which make this family useful to us. If $I = (x - a/2, x + a/2)$, $J = (y - b/2, y + b/2)$, we will abuse notation and write $\phi_{a, b}\colon I \to J$ where we really mean $T_y\phi_{a, b}T_{-x}\colon I \to J$. We will similarly abuse notation for intervals lying on $S^1$.

\begin{lemma}
\label{DefinesHomeo}
Suppose we are given a set of finite disjoint open intervals $\{I_i \subset M\colon i \in \Z\}$ such that $U = \cup_i I_i \subset M$ is dense. Let $l_i$ be the length of $I_i$ for each $i$. Let $l_i'$ be positive real numbers such that the following condition is satisfied: whenever $S \subset \Z$ is such that $\cup_{i \in S} I_i$ is contained in some bounded interval, $\Sigma_{i \in S} l_i' < \infty$. In particular, this holds if $\Sigma_{i \in \Z} l_i' < \infty$. Then there exists a homeomorphism $\phi\colon M \to \phi(M)$ such that the length of $\phi(I_i)$ is $l_i'$, and $m(\phi(C)) = 0$, where $C = U^c$ and $m$ is the Lebesgue measure.
\end{lemma}

\begin{proof}
Assume $M = \R$; the other cases are almost the same. Define $\phi(0) = 0$. Without loss of generality, $0 \in C$. Let $x > 0$ (the case $x < 0$ is similar). If $x \in C$, let $\phi(x) = \sum_{I_i \subset (0, x)} l_i'$. If $x \in I_j = (a_j, b_j)$, let $\phi(x) = \sum_{I_i \subset (0, x)} l_i' + (x - a_j)\frac{l_j'}{l_j}$. First, this really does define a function on the whole real line since by assumption $\sum_{I_i \subset (0, x)} l_i' < \infty$. This function $\phi$ is monotone: if $y > x$, then since $U \subset \R$ is dense there is (at least part of) an interval $I_i$ lying between $x$ and $y$. We claim it is also continuous. If $x$ lies in an interval, continuity at $x$ is obvious. Suppose $x \in C$, and $(x_n)_{n \geq 1} \subset C$ is a strictly increasing sequence of points approaching $x$ from the left with $x_1 > 0$. Note that $\phi(x) = \sum_{I_i \subset (0, x)} l_i' = \lim_{n \to \infty} \sum_{I_i \subset (0, x_n)} l_i'$, since $I_i \subset (0, x)$ implies $I_i \subset (0, x_n)$ for some $n$. But the right-hand side is just $\lim_{n \to \infty} \phi(x_n)$. Therefore, $\phi$ is a monotone-increasing, continuous function, hence a homeomorphism onto its image. Notice that $m(\phi(C)) = 0$, since for $y > x$, $\phi(y) - \phi(x)$ depends only on intervals lying between $x$ and $y$.
\end{proof}

Recall that we write $I_{i, j}$ (resp. $M_{i, j}$) for the open intervals making up the sets $I_i$ (resp. $M_i$) in $\mathcal{I}$ (resp. $\mathcal{M}$). The numbering is arbitrary. After applying a topological conjugacy, we want these intervals (which by abuse of notation we will still write $I_{i, j}$ or $M_{i, j}$) to have certain lengths, which we now define. Let $d = d(G)$. Recall that we are trying to conjugate $G$ to a group of $C^{1 + \alpha}$ diffeomorphisms, where $0 < \alpha < 1/d$.

Define $$|I_{i, j}|_S = 1 + \min\{|g|_S\colon (g\colon I_{i, 0} \to I_{i, j})\}, |M_{i, j}|_S = 1 + \min\{|g|_S\colon (g\colon I_{i, 0} \to M_{i, j})\}.$$

We will make the length of $I_{i, j}$ after conjugacy $$\ell_{i, j} = \frac{1}{(2^{|i|\alpha} + |I_{i, j}|_S)^{1/\alpha}},$$ and in exactly the same way we will make the length of $M_{i, j}$ after conjugacy $$\ell_{i, j}' = \frac{1}{(2^{|i|\alpha} + |M_{i, j}|_S)^{1/\alpha}}.$$

\begin{proposition}
The sum $\sum_{i, j} \ell_{i, j} + \ell_{i, j}' < \infty$.
\end{proposition}

\begin{proof}
 Let us consider $\sum_{i, j}\ell_{i, j}$, the other sum being similar. Fix $i$. Note that $\#\{g \in G \colon |g|_S = n\} \sim n^{d - 1}$. Therefore, $\#\{j\colon |I_{i, j}|_S = n\} \lesssim n^{d - 1}$, since $|I_{i, j}|_S = n$ implies that there is $g$ of length $n$ such that $g\colon I_{i, 0} \to I_{i, j}$, but there may be more than one such $g$. Thus modulo constants, the sum $\sum_j \ell_{i, j}$ is bounded by $\sum_{n \geq 0} \frac{n^{d - 1}}{(2^{|i|\alpha} + n)^{1/\alpha}}$. This sum clearly converges, since it has the form $\sum 1/n^{1/\alpha - d + 1}$ and $1/\alpha - d > 0$. In fact, the reader can check that the value to which it converges decays exponentially as $|i| \to \infty$, so $\sum_{i, j} \ell_{i, j} < \infty$.
\end{proof}

Therefore, by Lemma \ref{DefinesHomeo}, there exists a conjugacy after which the intervals have these lengths and their complement has measure 0.

We apply a further topological conjugacy, after which the action of $G$ has the following form. The unique map from $I_{i, 0}$ to $I_{i, j}$ under the action of $G$ is defined to be $\phi_{\ell_{i, 0}, \ell_{i, j}}$ (so the map from $I_{i, j}$ to $I_{i, k}$ is $\phi_{\ell_{i, j}, \ell_{i, k}}$). For each $j$, we arbitrarily choose $g_{i, j} \in G \colon M_{i, 0} \to M_{i, j}$ such that $|g_{i, j}|_S = |M_{i, j}|_S - 1$ (so $g_{i, j}$ is an element of minimum length among those sending $M_{i, 0}$ to $M_{i, j}$). We define $g_{i, j}$ to act via $\phi_{\ell_{i, 0}', \ell_{i, j}'}$.

Finally, we must describe how $G$ acts on $M_{i, 0}$ after conjugacy. Let $G_i = G_{M_{i, 0}} = \{g|_{M_{i, 0}}\colon g \in G, g(M_{i, 0}) = M_{i, 0}\}$. Recall from before that $G_i$ is abelian. We ask that every element of $G_i$ have the form $\phi_{\ell_{i, 0}'}T_a\phi_{\ell_{i, 0}'}^{-1}, a \in \R$. %If the generators of $G_i$ are $h_1, \ldots, h_k$, let us write $h_n = \phi_{\ell_{i, 0}'}T_{a_n}\phi_{\ell_{i, 0}'}^{-1}$ for $1 \leq n \leq k$.

%We can write $g(j(n))^{-1}g_n = \phi_{\ell_{i, 0}'}T_{\alpha(n)}\phi_{\ell_{i, 0}'}^{-1}$; we further require that for every $1\leq n \leq k$, $\alpha(n) \leq \ell_{i, 0}'$. We claim that, after conjugating $N$ into the given form, it is a subgroup of $\Diff_+^1(\R)$. We need the following lemma.

We claim that the conjugated action of $G$ as we have now described it is of class $C^{1 + \alpha}$. To prove this, we need the following lemma, which will tell us that we need only check the action is $C^{1 + \alpha}$ on the intervals $I_{i, j}$ and $M_{i, j}$.

\begin{lemma}[\cite{Navas}, Lemma 4.1.22]
\label{Patching}
Let $\{I_n\colon n \in \N\}$ be a family of closed intervals in $[0, 1]$ (resp. in $S^1$) having disjoint interiors and such that the complement of their union has zero Lebesgue measure. Suppose that $\phi$ is a homeomorphism of $[0, 1]$ (resp. $S^1$) such that its restrictions to each interval $I_n$ are $C^{1 + \alpha}$ diffeomorphisms which are $C^1$-tangent to the identity at the endpoints of $I_n$ and whose derivatives have $\alpha$-norms bounded from above by a constant $C$. Then $\phi$ is a $C^{1 + \alpha}$ diffeomorphism of the whole interval $[0, 1]$ (resp. $S^1$), and the $\alpha$-norm of its derivative is less than or equal to $2C$.
\end{lemma}

\begin{comment}
\begin{lemma}
\label{IsC1}
Let $C \subset M$ be a closed set of measure 0. If $f\colon M \to M$ is a homeomorphism such that $f(C) = C$, and $f$ is $C^1$ on $M\setminus C$ with $f'(x_n) \to 1$ for any sequence of points $x_n \in M\setminus C$ with $x_n \to x \in C$, then $f$ is $C^1$ on $M$ and has derivative 1 on $C$.
\end{lemma}

\begin{proof}
Let $x \in C$. Suppose $y > x$ is very close to $x$. (If $M = S^1$, choose an orientation.) The set $(x, y)$ is the union of $C \cap (x, y)$ and a countable collection of open intervals $I_i$. $f(I_i)$ is an interval of length approximately $\ell(I_i)$, since $f' \approx 1$ on $(x, y) \setminus C$. Now $f((x, y)) = f(C \cap (x, y)) \cup \bigcup_i f(I_i)$; letting $m$ be the Lebesgue measure, we have $m(\bigcup_i f(I_i)) \approx m(\bigcup_i I_i) = y - x$, since $C$ has measure 0. Moreover, $m(f(C \cap (x, y))) = 0$ since $C$ has measure 0 and $f(C \cap (x, y)) \subset C$. So $f(x, y)$ is an interval of measure (i.e. length) approximately $y - x$, which implies that $\frac{f(y) - f(x)}{y - x} \approx 1$, and so $f'(x) = 1$.
\end{proof}
\end{comment}

Before we finish the proof of Theorem \ref{Smoothability}, we need a definition. Let $G$ be a finitely-generated group, and $H$ a finitely-generated subgroup. Let $S, T$ be finite generating sets of $G$ and $H$, respectively. We define the \emph{distortion} to be $$\dist(H, T; G, S)(n) = \diam_T(B_S(n) \cap H),$$ where $B_S(n)$ is the ball of radius $n$ about the identity in $G$ measured in the word metric coming from $S$, and $\diam_T$ is the diameter measured in the word metric coming from $T$. Although this definition depends on the chosen generating sets $S$ and $T$, its growth type does not; as for growth, if $S'$ and $T'$ are different generating sets of $G$ and $H$ respectively, then $\dist(H, T'; G, S')(n) \sim \dist(H, T; G, S)(n)$, so we can speak unambiguously about the equivalence class $\dist(H, G)(n)$. For background, see e.g. \cite{Gromov}.

If $G$ is a finitely-generated virtually nilpotent group, and $H$ is any subgroup, we have $\dist(H, G)(n) \lesssim \mathcal{G}_G(n)$, with $\dist(H, G)(n) \sim \mathcal{G}_G(n)$ only if $G$ is finite or has $\Z$ as a finite-index subgroup. Otherwise, $G$ contains a finite-index torsion-free nilpotent subgroup $N$ which is not trivial or $\Z$. It is elementary to show that the distortion of $H \subset G$ is equivalent to the distortion of $H \cap N \subset N$.

If $N \cong \Z^m$, for any nontrivial subgroup $H \subset N$ we will have $\dist(H, N)(n) \sim n$, which is equivalent to $\mathcal{G}_{\Z^m}(n)$ only if $m = 1$. Otherwise, let $k > 1$ be the nilpotence class of $N$. We have $d(N) = \Sigma_{i \geq 1} i\cdot r_i$. It is a well-known fact (see for example 5.2.6. of \cite{Robinson}) that the abelianization of an infinite nilpotent group is infinite; if $N$ is finitely-generated, then the torsion-free rank of $G/[G, G]$ must be at least 1. On the other hand, the torsion-free rank of $N^k$ is at least 1 since $N$ is torsion-free and $N^k$ is nontrivial. Therefore, $d(N) \geq 1 + k$. On the other hand, for any $H \subset N$, $\dist(H, N)(n) \lesssim n^k$ (for an exact expression of $\dist(H, N)(n)$, see \cite{Osin}).

\begin{proof}[Proof of Theorem \ref{Smoothability}.]
By Lemma \ref{Patching}, all we need to show is that the generators $g_n$ for $G$ are $C^{1 + \alpha}$ with uniformly bounded $C^\alpha$ norm on the intervals $I_{i, j}$, $M_{i, j}$.

On the $I_{i, j}$, this is easy. Indeed, by construction $|g_n(I_{i, j})|_S$ can differ from $|I_{i, j}|_S$ by at most 1. The reader can check, with the interval lengths as we have defined them, together with equation (1), that we do get a uniform bound on the $C^\alpha$ norm of $g_n'$ on $I_{i, j}$.

%Therefore, $|1 - \frac{\ell(g_n(I_{i, j}))}{\ell_{i, j}}|$ is $O(\ell_{i, j}^{\frac{1}{d + 1}})$ as $\ell_{i, j} \to 0$, and since $g_n$ goes from $I_{i, j}$ to $g_n(I_{i, j})$ by a member of our equivariant family $\{\phi_{\alpha, \beta}\}$, $\max_{x \in I_{i, j}} |g_n'(x) - 1|$ is also $O(\ell_{i, j}^{\frac{1}{d + 1}})$.

Now suppose there are some minimal intervals $M_{i, j}$, so in particular $G$ is not finite and does not have a copy of $\Z$ as a finite-index subgroup. Consider the action of $g_n$ on an interval $M_{i, j}$. Say $g_n\colon M_{i, j}\to M_{i, j'}$. Note that, if $g_{i, j}$ and $g_{i, j'}$ are our chosen maps from $M_{i, 0}$ to $M_{i, j}$ and $M_{i, j'}$ respectively, then $$g_n = g_{i, j'}(g_{i, j'}^{-1}g_ng_{i, j})g_{i, j}^{-1}.$$

Here, $g_{i, j'}^{-1}g_ng_{i, j}$ sends $M_{i, 0}$ to itself. It is a word of length at most $2|M_{i, j}|_S + 2$, since $|g_{i, j'}|_S \leq |M_{i, j}|_S + 1$ and $|g_n|_S = 1$. Recall from above that the degree of distortion for subgroups of $G$ is at most the nilpotence class $k$ of a finite-index torsion-free nilpotent subgroup of $G$, which is strictly less than $d = d(G)$. The length of $g_{i, j'}^{-1}g_ng_{i, j}$ in terms of a generating set for $G_i$ will be at most of order $|M_{i, j}|_S^k$, and hence $\phi_{\ell_{i, 0}'}^{-1}g_{i, j'}^{-1}g_ng_{i, j}\phi_{\ell_{i, 0}'}$ will be translation by an amount of order at most $|M_{i, j}|_S^k$. Therefore, $g_n$ has the form $\phi_{b'}T_a\phi_b^{-1}$, where $$a \leq c_i|M_{i, j}|_S^k, b = \ell_{i, j}', b' = \ell_{i, j'}',$$ $c_i$ being a constant that depends on the speed of translation of the generators of $G_i$ on $M_{i, 0}$. By applying the appropriate conjugacy, we can make the translation speeds and hence the $c_i$ as small as we like.

%\frac{1}{(2^{|i|\alpha} + |M_{i, j}|_S)^{1/\alpha}}, \text{ and } b' = \frac{1}{(2^{|i|\alpha} + |M_{i, j'}|_S)^{1/\alpha}}

We have $g_n\colon M_{i, j} \to M_{i, j'}$ given by $$g_n(x) = \frac{b'}{\pi}\arctan(\frac{b'}{b}\tan(\frac{\pi x}{b}) + b'a).$$

This has derivative $$g_n'(x) = \frac{b'^2}{b^2}\cdot\frac{1 + \tan^2(\frac{\pi x}{b})}{1 + (\frac{b'}{b}\tan(\frac{\pi x}{b}) + b'a)^2},$$

and second derivative $$g_n''(x) = \frac{2\pi\frac{b'^2}{b^3}\sec^4(\frac{\pi x}{b})\big(\frac{b'^2}{b}a\sin^2(\frac{\pi x}{b}) + (b'^2a^2 - \frac{b'^2}{b^2} + 1)\sin(\frac{\pi x}{b})\cos(\frac{\pi x}{b}) - \frac{b'^2}{b}a\cos^2(\frac{\pi x}{b})\big)}{(1 + (\frac{b'}{b}\tan(\frac{\pi x}{b}) + b'a)^2)^2}.$$

By analyzing this second derivative, we will show that we get the desired bound on the $C^\alpha$ norm of $g_n'$ on the intervals $M_{i, j}$.

Consider the function $\frac{\sec^4(\frac{\pi x}{b})}{(1 + (\frac{b'}{b}\tan(\frac{\pi x}{b}) + b'a)^2)^2}$ on $M_{i, j}$. As we vary $i$ and $j$, we vary the constants $a$, $b$, and $b'$; observe that for all $x$ and for all $i, j$ these functions are uniformly bounded, provided that we make the $c_i$ small enough. This is because whenever $|i|$ or $|M_{i, j}|_S$ is large we will have $\frac{b'}{b} \approx 1$ and $b'a \approx 0$.

Obviously, the sine and cosine functions in the numerator have magnitude at most 1, so we only need to consider their coefficients as we vary $i$ and $j$, these being (disregarding a factor of $2\pi$) $\frac{b'^4}{b^4}a$ and $\frac{b'^2}{b^3}(b'^2a^2 - \frac{b'^2}{b^2} + 1)$. We claim that there exists a constant $C$ such that these are both bounded above by $$C\cdot b^{\alpha - 1} = C\cdot (2^{|i|\alpha} + |M_{i, j}|_S)^{1/\alpha - 1}$$ for all $i, j$.

First, consider $\frac{b'^4}{b^4}a$. Whenever $|i|$ or $|M_{i, j}|_S$ is large, $\frac{b'^4}{b^4}$ will be very close to 1. Moreover, we know that $$a \leq c_i|M_{i, j}|_S^k \leq (2^{|i|\alpha} + |M_{i, j}|_S)^{1/\alpha - 1}$$ provided $c_i$ is small enough (since $1/\alpha - 1 > k$).

Now, consider $\frac{b'^2}{b^3}(b'^2a^2)$. This is equal to $\frac{b'^4}{b^4}\cdot ba^2$. Again, when $|i|$ or $|M_{i, j}|_S$ is large, $\frac{b'^4}{b^4}$ will be very close to 1, so we are left to consider $ba^2$. We have $$b = \frac{1}{(2^{|i|\alpha} + |M_{i, j}|_S)^{1/\alpha}} \text{ and } a \leq c_i|M_{i, j}|_S^k,$$ so 

\begin{align*}
ba^2 &\leq \frac{c_i^2|M_{i, j}|_S^{2k}}{(2^{|i|\alpha} + |M_{i, j}|_S)^{1/\alpha}} \\
&< \frac{c_i^2(2^{|i|\alpha} + |M_{i, j}|_S)^{2k}}{(2^{|i|\alpha} + |M_{i, j}|_S)^{1/\alpha}} \\
&= c_i^2(2^{|i|\alpha} + |M_{i, j}|_S)^{2k - 1/\alpha} \\
&< (2^{|i|\alpha} + |M_{i, j}|_S)^{1/\alpha - 1}
\end{align*} for appropriate choice of $c_i$, since $2k - 1/\alpha < 1/\alpha - 1$.

Finally, consider $\frac{b'^2}{b^3}(-\frac{b'^2}{b^2} + 1)$. If $|M_{i, j'}|_S = |M_{i, j}|_S$, then this term is 0. Otherwise, $$b = \frac{1}{(2^{|i|\alpha} + |M_{i, j}|_S)^{1/\alpha}} \text{ and } b' = \frac{1}{(2^{|i|\alpha} + |M_{i, j}|_S \pm 1)^{1/\alpha}}.$$

As before, for large $|i|$ or $|M_{i, j}|_S$, $\frac{b'^2}{b^2}$ is close to 1, so we can disregard it. Therefore, we are left with $$\frac{1}{b}(-\frac{b'^2}{b^2} + 1) = (2^{|i|\alpha} + |M_{i, j}|_S)^{1/\alpha}(-(\frac{2^{|i|\alpha} + |M_{i, j}|_S}{2^{|i|\alpha} + |M_{i, j}|_S \pm 1})^{2/\alpha} + 1).$$ The reader can check that as $2^{|i|\alpha} + |M_{i, j}|_S \to \infty$, the quantity $$(2^{|i|\alpha} + |M_{i, j}|_S)(-(\frac{2^{|i|\alpha} + |M_{i, j}|_S}{2^{|i|\alpha} + |M_{i, j}|_S \pm 1})^{2/\alpha} + 1)$$ approaches $\pm 2/\alpha$, so for large enough $C$, we will have $$\frac{b'^2}{b^3}(-\frac{b'^2}{b^2} + 1) \leq C\cdot (2^{|i|\alpha} + |M_{i, j}|_S)^{1/\alpha - 1}$$ for all $i, j$.

%$|M_{i, j}|_S^{1/\alpha - 1}$ (meaning that there is some $C > 0$ such that their magnitude is actually bounded by $C|M_{i, j}|_S^{1/\alpha - 1}$).

%If we disregard the factor of $2\pi$, then the coefficient of $\sin^2(\frac{\pi x}{b})$ is $\frac{b'^4}{b^4}a$. We know that $\frac{b'^4}{b^4}$ is 1 plus lower order, and $a$ is of order at most $|M_{i, j}|_S^k < |M_{i, j}|_S^{1/\alpha - 1}$. The coefficient of $\cos^2(\frac{\pi x}{b})$ is the same.

%Finally, the coefficient of $\sin(\frac{\pi x}{b})\cos(\frac{\pi x}{b})$ is $\frac{b'^2}{b^3}(b'^2a^2 - \frac{b'^2}{b^2} + 1)$. The term $\frac{b'^4}{b^3}a^2$ is of order at most $$|M_{i, j}|_S^{-1/\alpha}\cdot|M_{i, j}|_S^{2k} < |M_{i, j}|_S^{1/\alpha - 1}.$$ Consider the term $\frac{b'^2}{b^3}(-\frac{b'^2}{b^2} + 1)$. If $|M_{i, j'}|_S = |M_{i, j}|_S$, then this term is 0. Otherwise, $$b = \frac{1}{(2^{|i|\alpha} + |M_{i, j}|_S)^{1/\alpha}} \text{ and } b' = \frac{1}{(2^{|i|\alpha} + |M_{i, j}|_S \pm 1)^{1/\alpha}}.$$ In this case, $$\frac{b'^2}{b^2} = 1 \mp \frac{2/\alpha}{|M_{i, j}|_S} + \text{lower order}$$ in $|M_{i, j}|_S$, so $$-\frac{b'^2}{b^2} + 1 = \pm\frac{2/\alpha}{|M_{i, j}|_S} + \text{lower order}.$$ Thus $\frac{b'^2}{b^3}(-\frac{b'^2}{b^2} + 1)$ is of order $$|M_{i, j}|_S^{1/\alpha}\cdot \frac{2/\alpha}{|M_{i, j}|_S} \sim |M_{i, j}|_S^{1/\alpha - 1}.$$

%Therefore, for some $C > 0$ we have $|g_n''| \leq C\cdot |M_{i, j}|_S^{1/\alpha - 1}$ for all $j$, and we claim this implies a uniform bound on the $C^\alpha$ norm of $g_n'|_{M_{i, j}}$ for every $j$. Let $x, y \in M_{i, j}$. Then we have 

Therefore, the $C^\alpha$ norm of $g_n'$ on $M_{i, j}$ is at most $C$, for any $i, j$:

\begin{align*}
\frac{|g_n'(y) - g_n'(x)|}{|y - x|^\alpha} &\leq \frac{C(2^{|i|\alpha} + |M_{i, j}|_S)^{1/\alpha - 1}\cdot |y - x|}{|y - x|^\alpha} \\
&\leq \frac{C(2^{|i|\alpha} + |M_{i, j}|_S)^{1/\alpha - 1}\cdot \frac{1}{(2^{|i|\alpha} + |M_{i, j}|_S)^{1/\alpha}}}{(\frac{1}{(2^{|i|\alpha} + |M_{i, j}|_S)^{1/\alpha}})^\alpha} \\
&= C.
\end{align*}

Combining this with the fact that we have such a bound also on the $I_{i, j}$, and applying Lemma \ref{Patching}, we conclude that $g_n$ is $C^{1 + \alpha}$ on $M$, as desired.

%$M_{i, j}$ for all $i, j$. Combining this with the fact that we have such a bound also on the $I_{i, j}$, and applying Lemma \ref{Patching}, we conclude that $g_n$ is $C^{1 + \alpha}$ on $M$, as desired.

\begin{comment}
We claim that if $\alpha\beta$ is small, then $\phi_\beta T_\alpha\phi_\beta^{-1}$ has derivative uniformly close to 1. In fact, $$\phi_\beta T_\alpha\phi_\beta^{-1}(x) = \frac{\beta}{\pi}\arctan(\alpha\beta + \tan(\frac{\pi x}{\beta})),$$ which has derivative $$\frac{d}{dx}\phi_\beta T_\alpha\phi_\beta^{-1} = \frac{1 + \tan^2(\frac{\pi x}{\beta})}{1 + (\alpha\beta + \tan(\frac{\pi x}{\beta}))^2},$$ which has reciprocal $1 + \frac{2\alpha\beta\tan(\frac{\pi x}{\beta}) + \alpha^2\beta^2}{1 + \tan^2(\frac{\pi x}{\beta})}$. Obviously $\frac{\alpha^2\beta^2}{1 + \tan^2(\frac{\pi x}{\beta})}$ is $O((\alpha\beta)^2)$ as $\alpha\beta \to 0$. Moreover, $\frac{\tan(\frac{\pi x}{\beta})}{1 + \tan^2(\frac{\pi x}{\beta})}$ has magnitude bounded by $1/2$, so $|\frac{2\alpha\beta\tan(\frac{\pi x}{\beta})}{1 + \tan^2(\frac{\pi x}{\beta})}| \leq \alpha\beta$ for all $x$.

Therefore, $g_{i, j}[g_{i, j'}^{-1}g_ng_{i, j}]g_{i, j}^{-1}$ has derivative close to 1 for $|M_{i, j}|_S$ large enough. If we choose the translation speeds of the generators of $G_i$ to be small enough, then $g_{i, j}[g_{i, j'}^{-1}g_ng_{i, j}]g_{i, j}^{-1}$ will have derivative close to 1 everywhere on $M_i$. And as we have seen, $g_{i, j'}g_{i, j}^{-1}$ has derivative close to 1, since $\ell_{i, j}' \approx \ell_{i, j'}'$. This proves the theorem.
\end{comment}
\end{proof}

\begin{proof}[Proof of Corollary \ref{Corollary}.]
The proof is immediate, by choosing the lengths of the intervals appropriately and making the minimal abelian actions on the $M_{i, j}$ close enough to the identity in the $C^1$ topology.
\end{proof}

\end{document}